\documentclass[11pt]{article}
\usepackage[a4paper,margin=1.05in]{geometry}
\usepackage{amsmath,amssymb,amsthm}
\usepackage{enumitem}
\usepackage{xcolor}
\usepackage[colorlinks=true,linkcolor=blue!55!black,citecolor=blue!55!black,urlcolor=blue!55!black]{hyperref}

\usepackage{authblk}

\theoremstyle{plain}
\newtheorem{theorem}{Theorem}[section]
\newtheorem{lemma}[theorem]{Lemma}
\newtheorem{proposition}[theorem]{Proposition}
\newtheorem{corollary}[theorem]{Corollary}
\theoremstyle{remark}
\newtheorem{remark}[theorem]{Remark}

\DeclareMathOperator{\rank}{rank}
\DeclareMathOperator{\col}{col}
\DeclareMathOperator{\spn}{span}
\DeclareMathOperator{\Log}{Log}
\DeclareMathOperator{\Hom}{Hom}
\newcommand{\Z}{\mathbb{Z}}
\newcommand{\C}{\mathbb{C}}
\newcommand{\R}{\mathbb{R}}
\newcommand{\Hn}{\mathcal{H}_n}
\newcommand{\Gn}{G_n}
\newcommand{\Ghat}{\widehat{G}_n}
\newcommand{\uu}{\mathbf{u}}
\newcommand{\qq}{\mathbf{q}}
\newcommand{\NN}{\mathbf{N}}
\newcommand{\UU}{\mathbf{U}}
\newcommand{\res}[1]{\langle #1\rangle_n}
\newcommand{\ip}[2]{\langle #1,\,#2\rangle}

\title{Exact Rank of the Largest Good Lattice Point Matrix}
\author[1]{Long-Hao Xu\footnote{Corresponding author: {e-mail: longhao.xu@outlook.com}}}

\affil[1]{Department of Medical Statistics, University Medical Center G\"{o}ttingen, G\"{o}ttingen, Germany}

\begin{document}
\maketitle

\begin{abstract}
We prove that for every integer $n>2$ the largest good lattice point matrix
$\UU_n=[\uu_h:h\in\Hn]$, whose columns are generated by all integers $h<n$ coprime
to $n$, has rank exactly $\varphi(n)/2+1$. The upper bound follows from the pairing
identity $\uu_h+\uu_{n-h}=(n,\dots,n,2n)^{\top}$. For the lower bound the columns are
centered and the resulting orbit under the multiplicative group $(\Z/n\Z)^{\times}$ is
analysed by means of odd Dirichlet characters; the essential arithmetic input is the
non-vanishing of $\sum_{a}a\,\psi(a)$ for odd primitive characters $\psi$, which is
equivalent to Dirichlet's theorem $L(1,\psi)\neq 0$.
\end{abstract}

\section{Introduction}

Throughout, $n>2$ is a fixed integer. Let
\[
\Gn=(\Z/n\Z)^{\times},\qquad
\Hn=\{\,h\in\{1,\dots,n-1\}:\gcd(h,n)=1\,\},
\]
so that $\Hn$ is a set of representatives of $\Gn$ and $|\Hn|=\varphi(n)$.
For $x\in\Z$ let $\res{x}$ denote the unique element of $\{1,\dots,n\}$ that is
congruent to $x$ modulo $n$; explicitly
\[
\res{x}=x-n\Bigl\lfloor\frac{x-1}{n}\Bigr\rfloor .
\]
Thus $\res{x}=n$ if and only if $n\mid x$. For $h\in\Hn$ the good lattice point (GLP)
column generated by $h$ is
\[
\uu_h=\bigl(\res{h},\res{2h},\dots,\res{nh}\bigr)^{\top}\in\R^{n},
\]
and the largest $n$-run GLP matrix is
\[
\UU_n=[\uu_h:h\in\Hn]\in\R^{n\times\varphi(n)},
\]
the order of the columns being irrelevant for what follows.

\begin{theorem}\label{thm:main}
For every integer $n>2$,
\[
\rank(\UU_n)=\frac{\varphi(n)}{2}+1 .
\]
\end{theorem}

\paragraph{Convention on the ground field.}
The rank of a real matrix equals the size of its largest non-vanishing minor and is
therefore the same whether the matrix is regarded over $\R$ or over $\C$. Equivalently,
real vectors are linearly independent over $\R$ if and only if they are linearly
independent over $\C$. We use this freely: the lower bound is proved over $\C$.

\paragraph{Outline.}
The proof has four ingredients: the pairing identity (Section~\ref{sec:pair}), which
gives the upper bound; the passage to centered columns, which turns the problem into a
statement about a permutation representation of $\Gn$ (Section~\ref{sec:action}); a
non-vanishing statement for the sums $\sum_a a\,\psi(a)$ attached to odd primitive
Dirichlet characters $\psi$, which rests on Dirichlet's theorem $L(1,\psi)\neq0$
(Section~\ref{sec:nt}); and an explicit construction of eigenvectors that pair
non-trivially with the centered column (Section~\ref{sec:lower}).

\section{The pairing identity and the upper bound}\label{sec:pair}

\begin{lemma}[Pairing identity]\label{lem:pair}
Let $h\in\Hn$. Then $n-h\in\Hn$, $n-h\neq h$, and
\[
\uu_h+\uu_{n-h}=\NN_n:=(n,\dots,n,2n)^{\top}.
\]
\end{lemma}

\begin{proof}
Since $\gcd(n-h,n)=\gcd(h,n)=1$, we have $n-h\in\Hn$. If $n-h=h$, then $h=n/2$ and
$\gcd(h,n)=n/2>1$ because $n>2$, a contradiction.

Let $1\le i\le n-1$. Since $\gcd(h,n)=1$ and $n\nmid i$, we have $n\nmid ih$, so
$r:=\res{ih}\in\{1,\dots,n-1\}$. Then
$i(n-h)\equiv -ih\equiv -r\equiv n-r\pmod n$ with $n-r\in\{1,\dots,n-1\}$, hence
$\res{i(n-h)}=n-r$, and the $i$-th coordinates of $\uu_h$ and $\uu_{n-h}$ add up to
$n$. For $i=n$ we have $nh\equiv n(n-h)\equiv 0\pmod n$, so both $n$-th coordinates
equal $n$ and add up to $2n$.
\end{proof}

By Lemma~\ref{lem:pair} the set $\Hn$ is partitioned into $\varphi(n)/2$ two-element
subsets $\{h,n-h\}$. Throughout we fix a set
\begin{equation}\label{eq:R}
R\subset\Hn \ \text{ containing exactly one element of each pair } \{h,n-h\},
\qquad |R|=\varphi(n)/2 .
\end{equation}

\begin{corollary}[Upper bound]\label{cor:upper}
$\rank(\UU_n)\le \varphi(n)/2+1$.
\end{corollary}

\begin{proof}
For $h\in R$ we have $\uu_{n-h}=\NN_n-\uu_h$ by Lemma~\ref{lem:pair}. Hence every
column of $\UU_n$ lies in $\spn(\{\NN_n\}\cup\{\uu_h:h\in R\})$, a space of dimension
at most $1+|R|=1+\varphi(n)/2$.
\end{proof}

\section{Centered columns and the action of \texorpdfstring{$\Gn$}{Gn}}\label{sec:action}

\subsection{Centered columns}

Define $q_n:\Z/n\Z\to\R$ by
\[
q_n(x)=
\begin{cases}
0, & x\equiv 0\pmod n,\\[3pt]
r-\dfrac{n}{2}, & x\equiv r\pmod n,\ 1\le r\le n-1 ,
\end{cases}
\]
and for $h\in\Hn$ put
\[
\qq_h:=\uu_h-\tfrac12\NN_n\in\R^{n}.
\]

\begin{lemma}\label{lem:centered}
\begin{enumerate}[label=\textup{(\roman*)},leftmargin=*]
\item $q_n(-x)=-q_n(x)$ for every $x\in\Z/n\Z$.
\item For $h\in\Hn$ and $1\le i\le n$ one has $(\qq_h)_i=q_n(ih)$; in particular
      $(\qq_h)_n=0$.
\item For $h\in\Hn$ one has $\qq_{n-h}=-\qq_h$.
\end{enumerate}
\end{lemma}

\begin{proof}
(i) If $x\equiv 0$ both sides vanish. If $x\equiv r$ with $1\le r\le n-1$, then
$-x\equiv n-r$ with $1\le n-r\le n-1$, so
$q_n(-x)=(n-r)-\frac n2=-\bigl(r-\frac n2\bigr)=-q_n(x)$.

(ii) For $1\le i\le n-1$ we have $ih\not\equiv 0\pmod n$ (as in the proof of
Lemma~\ref{lem:pair}), hence $(\qq_h)_i=\res{ih}-\frac n2=q_n(ih)$. For $i=n$,
$(\qq_h)_n=n-\frac12\cdot 2n=0=q_n(0)=q_n(nh)$.

(iii) By Lemma~\ref{lem:pair}, $\qq_h+\qq_{n-h}=\uu_h+\uu_{n-h}-\NN_n=0$.
\end{proof}

\subsection{The permutation representation}

Let $V$ be the $\C$-vector space of all functions $f:\Z/n\Z\to\C$, equipped with the
inner product
\[
\ip{f}{g}=\sum_{x\in\Z/n\Z} f(x)\,\overline{g(x)},
\]
which is linear in the first and conjugate-linear in the second variable. For
$h\in\Gn$ define $\rho(h):V\to V$ by
\[
(\rho(h)f)(x)=f(hx),\qquad x\in\Z/n\Z .
\]

\begin{lemma}\label{lem:rep}
$\rho(1)=\mathrm{id}_V$ and $\rho(hk)=\rho(h)\rho(k)$ for all $h,k\in\Gn$; each
$\rho(h)$ is unitary; and
\[
\ip{\rho(h)f}{g}=\ip{f}{\rho(h^{-1})g}\qquad(f,g\in V,\ h\in\Gn).
\]
\end{lemma}

\begin{proof}
Clearly $\rho(1)=\mathrm{id}_V$, and
$(\rho(h)\rho(k)f)(x)=(\rho(k)f)(hx)=f(khx)=(\rho(hk)f)(x)$, where we used that $\Gn$
is abelian. Since $x\mapsto hx$ is a bijection of $\Z/n\Z$, the operator $\rho(h)$
permutes the coordinates of $f$ and therefore preserves the inner product. Finally,
unitarity applied to $f$ and $\rho(h^{-1})g$ gives
$\ip{\rho(h)f}{\rho(h)\rho(h^{-1})g}=\ip{f}{\rho(h^{-1})g}$, and
$\rho(h)\rho(h^{-1})=\rho(1)=\mathrm{id}_V$.
\end{proof}

We identify $V$ with $\C^{n}$ via
\[
\iota:V\to\C^{n},\qquad
\iota(f)=\bigl(f(1),f(2),\dots,f(n-1),f(0)\bigr)^{\top},
\]
i.e.\ the coordinate $x\in\{1,\dots,n-1\}$ is placed in row $x$ and the coordinate
$x=0$ in row $n$. Regard $q_n$ as an element of $V$. For $h\in\Hn$ and
$1\le i\le n-1$, $\iota(\rho(h)q_n)_i=q_n(hi)=(\qq_h)_i$ by
Lemma~\ref{lem:centered}(ii), and $\iota(\rho(h)q_n)_n=q_n(0)=0=(\qq_h)_n$. Hence
\begin{equation}\label{eq:ident}
\iota\bigl(\rho(h)q_n\bigr)=\qq_h\qquad(h\in\Hn).
\end{equation}
Moreover, by Lemma~\ref{lem:centered}(i),
\begin{equation}\label{eq:odd}
\rho(-h)q_n=-\rho(h)q_n\qquad(h\in\Gn),
\end{equation}
since $(\rho(-h)q_n)(x)=q_n(-hx)=-q_n(hx)$. In particular $\rho(h)q_n$ depends only on
the class of $h$ in $\Gn$, and \eqref{eq:ident}, \eqref{eq:odd} are consistent with
Lemma~\ref{lem:centered}(iii).

\subsection{Characters of \texorpdfstring{$\Gn$}{Gn}}

Let $\Ghat=\Hom(\Gn,\C^{\times})$ be the character group of $\Gn$. We use the following
standard facts about the characters of a finite abelian group $G$ (see
e.g.\ \cite[Chapter~6]{Ap}): $|\widehat G|=|G|$; every character takes values in the
unit circle, so $\overline{\chi(g)}=\chi(g)^{-1}=\chi(g^{-1})$; and characters separate
points, i.e.\ for every $g\neq 1$ there is $\chi\in\widehat G$ with $\chi(g)\neq 1$.

Since $(-1)^2=1$ in $\Gn$, every $\chi\in\Ghat$ satisfies $\chi(-1)\in\{1,-1\}$. We
call $\chi$ \emph{odd} if $\chi(-1)=-1$ and \emph{even} if $\chi(-1)=1$. For a function
$c:\Gn\to\C$ we write
\[
\hat c(\chi)=\sum_{h\in\Gn}c(h)\,\chi(h)\qquad(\chi\in\Ghat).
\]

\begin{lemma}\label{lem:char}
\begin{enumerate}[label=\textup{(\alph*)},leftmargin=*]
\item For $g\in\Gn$, $\sum_{\chi\in\Ghat}\chi(g)=\varphi(n)$ if $g=1$, and
      $\sum_{\chi\in\Ghat}\chi(g)=0$ if $g\neq 1$.
\item If $c:\Gn\to\C$ satisfies $\hat c(\chi)=0$ for all $\chi\in\Ghat$, then
      $c\equiv 0$.
\item Exactly $\varphi(n)/2$ characters of $\Gn$ are odd.
\end{enumerate}
\end{lemma}

\begin{proof}
(a) For $g=1$ every term equals $1$ and there are $|\Ghat|=\varphi(n)$ terms. For
$g\neq 1$ choose $\chi_1$ with $\chi_1(g)\neq 1$. Since $\chi\mapsto\chi_1\chi$ is a
bijection of $\Ghat$,
\[
\sum_{\chi\in\Ghat}\chi(g)=\sum_{\chi\in\Ghat}(\chi_1\chi)(g)
=\chi_1(g)\sum_{\chi\in\Ghat}\chi(g),
\]
which forces $\sum_{\chi}\chi(g)=0$.

(b) For $k\in\Gn$, using $\overline{\chi(k)}=\chi(k^{-1})$ and (a),
\[
\sum_{\chi\in\Ghat}\overline{\chi(k)}\,\hat c(\chi)
=\sum_{h\in\Gn}c(h)\sum_{\chi\in\Ghat}\chi(h)\chi(k^{-1})
=\sum_{h\in\Gn}c(h)\sum_{\chi\in\Ghat}\chi(hk^{-1})
=\varphi(n)\,c(k).
\]
If $\hat c\equiv 0$ the left-hand side vanishes, hence $c(k)=0$ for all $k$.

(c) Since $n>2$ we have $-1\neq 1$ in $\Gn$, so there is $\chi_0\in\Ghat$ with
$\chi_0(-1)\neq 1$, i.e.\ $\chi_0(-1)=-1$. The map
$\varepsilon:\Ghat\to\{1,-1\}$, $\chi\mapsto\chi(-1)$, is a group homomorphism, and it
is surjective because $\varepsilon(\chi_0)=-1$. Its kernel (the set of even characters)
therefore has index $2$ in $\Ghat$, so exactly $|\Ghat|/2=\varphi(n)/2$ characters are
odd.
\end{proof}

\section{The number-theoretic input}\label{sec:nt}

Let $f\ge 1$. A Dirichlet character modulo $f$ is a character $\psi$ of
$(\Z/f\Z)^{\times}$, extended to a function on $\Z$ by $\psi(a)=0$ whenever
$\gcd(a,f)>1$. It is \emph{primitive} if it is not induced by a character of a smaller
modulus. The \emph{conductor} of a character $\chi$ modulo $n$ is the smallest $f\mid n$
such that $\chi$ is induced by a character $\chi^{*}$ modulo $f$, i.e.\
$\chi(h)=\chi^{*}(h)$ for all integers $h$ with $\gcd(h,n)=1$; this $\chi^{*}$ is
unique and primitive. The Gauss sum of $\psi$ is
\[
\tau(\psi)=\sum_{a=1}^{f}\psi(a)\,e^{2\pi i a/f}.
\]
We recall that $\sum_{a=1}^{f}\psi(a)=0$ for every non-principal $\psi$: choosing an
integer $b$ coprime to $f$ with $\psi(b)\neq 1$, the substitution $a\mapsto ba$ gives
$\sum_a\psi(a)=\psi(b)\sum_a\psi(a)$. We shall use the following three classical
facts.
\begin{enumerate}[label=\textup{(F\arabic*)},ref=(F\arabic*),leftmargin=*]
\item\label{F1} If $\psi$ is primitive modulo $f$, then for every integer $m$
\[
\sum_{a=1}^{f}\psi(a)\,e^{2\pi i a m/f}=\overline{\psi(m)}\,\tau(\psi),
\qquad\text{and}\qquad |\tau(\psi)|=\sqrt f .
\]
(See \cite[Chapter~9]{Da} or \cite[Chapter~8]{Ap}.)
\item\label{F2} For every complex number $z$ with $|z|=1$, $z\neq 1$, the series
$\sum_{m\ge 1}z^{m}/m$ converges and equals $-\Log(1-z)$, where $\Log$ denotes the
principal branch of the logarithm. (Convergence follows from Dirichlet's test, and
the value from Abel's limit theorem applied to the power series of $-\Log(1-z)$.)
\item\label{F3} (Dirichlet) If $\psi$ is a non-principal Dirichlet character, then
$L(1,\psi)=\sum_{m\ge 1}\psi(m)/m\neq 0$. (See \cite[Chapter~4]{Da} or
\cite[Chapter~6]{Ap}.)
\end{enumerate}

\begin{lemma}\label{lem:S}
Let $\psi$ be an odd primitive Dirichlet character modulo $f$, and put
\[
S(\psi):=\sum_{a=1}^{f-1}a\,\psi(a)=\sum_{\substack{1\le a<f\\ \gcd(a,f)=1}}a\,\psi(a).
\]
Then
\begin{equation}\label{eq:S}
L(1,\bar\psi)=-\frac{i\pi}{f\,\tau(\psi)}\,S(\psi),
\end{equation}
and in particular $S(\psi)\neq 0$.
\end{lemma}

\begin{proof}
Since $\psi$ is odd it is non-principal, so $f\ge 3$ (modulo $1$ and $2$ only the
principal character exists); thus $\psi(f)=0$ and the two expressions for $S(\psi)$
agree. Put $z_a=e^{2\pi i a/f}$ for $1\le a\le f-1$; then $|z_a|=1$ and $z_a\neq 1$.

By \ref{F1}, for every integer $m\ge 1$,
\[
\overline{\psi(m)}=\frac{1}{\tau(\psi)}\sum_{a=1}^{f-1}\psi(a)\,z_a^{\,m}
\]
(the term $a=f$ vanishes). Dividing by $m$ and summing over $1\le m\le M$, which is a
finite interchange of summations, we get
\[
\sum_{m=1}^{M}\frac{\overline{\psi(m)}}{m}
=\frac{1}{\tau(\psi)}\sum_{a=1}^{f-1}\psi(a)\sum_{m=1}^{M}\frac{z_a^{\,m}}{m}.
\]
As $M\to\infty$, the left-hand side tends to $L(1,\bar\psi)$ and, by \ref{F2}, each
inner sum on the right tends to $-\Log(1-z_a)$. Hence
\begin{equation}\label{eq:L1}
L(1,\bar\psi)=-\frac{1}{\tau(\psi)}\sum_{a=1}^{f-1}\psi(a)\Log(1-z_a).
\end{equation}
For $1\le a\le f-1$ we have $0<\pi a/f<\pi$ and
\[
1-z_a=1-e^{2\pi i a/f}=-2i\sin\Bigl(\frac{\pi a}{f}\Bigr)e^{i\pi a/f}
=2\sin\Bigl(\frac{\pi a}{f}\Bigr)\,e^{\,i(\pi a/f-\pi/2)},
\]
where $2\sin(\pi a/f)>0$ and $\pi a/f-\pi/2\in(-\pi/2,\pi/2)$. Therefore
\[
\Log(1-z_a)=\log\Bigl(2\sin\frac{\pi a}{f}\Bigr)+i\pi\Bigl(\frac{a}{f}-\frac12\Bigr).
\]
Now $\psi(f-a)=\psi(-1)\psi(a)=-\psi(a)$, while $\sin(\pi(f-a)/f)=\sin(\pi a/f)$;
pairing $a$ with $f-a$ therefore gives
$\sum_{a=1}^{f-1}\psi(a)\log\bigl(2\sin\frac{\pi a}{f}\bigr)=0$. Moreover
$\sum_{a=1}^{f-1}\psi(a)=0$ because $\psi$ is non-principal. Substituting into
\eqref{eq:L1} yields
\[
L(1,\bar\psi)=-\frac{i\pi}{\tau(\psi)}\sum_{a=1}^{f-1}\psi(a)\Bigl(\frac af-\frac12\Bigr)
=-\frac{i\pi}{f\,\tau(\psi)}\sum_{a=1}^{f-1}a\,\psi(a),
\]
which is \eqref{eq:S}. Finally $\tau(\psi)\neq 0$ by \ref{F1}, and $L(1,\bar\psi)\neq 0$
by \ref{F3} (note that $\bar\psi$ is again non-principal), so $S(\psi)\neq 0$.
\end{proof}

\begin{remark}\label{rem:B1}
In terms of the generalized Bernoulli number $B_{1,\psi}=\frac1f\sum_{a=1}^{f}a\,\psi(a)$
and the value $L(0,\psi)=-B_{1,\psi}$, Lemma~\ref{lem:S} reads
$S(\psi)=fB_{1,\psi}=-fL(0,\psi)\neq 0$; equivalently
$L(0,\psi)=\frac{\tau(\psi)}{i\pi}\,L(1,\bar\psi)$, which is the case $s=0$ of the
functional equation of $L(s,\psi)$ for odd primitive $\psi$. Either route reduces the
non-vanishing of $S(\psi)$ to Dirichlet's theorem \ref{F3}; cf.\
\cite[Chapter~4]{Wa}.
\end{remark}

\section{Eigenvectors and the lower bound}\label{sec:lower}

\begin{lemma}\label{lem:eig}
Let $\chi$ be an odd character of $\Gn$, let $f$ be its conductor, write $n=df$, and let
$\chi^{*}$ be the primitive character modulo $f$ inducing $\chi$. Define $e_\chi\in V$ by
\[
e_\chi(x)=
\begin{cases}
\chi^{*}(a), & \text{if } x\equiv da\pmod n \text{ for some integer } a
              \text{ with } \gcd(a,f)=1,\\[3pt]
0, & \text{otherwise.}
\end{cases}
\]
Then:
\begin{enumerate}[label=\textup{(\roman*)},leftmargin=*]
\item $e_\chi$ is well defined and $e_\chi\neq 0$;
\item $\rho(h)e_\chi=\chi(h)\,e_\chi$ for every $h\in\Gn$;
\item $\ip{q_n}{e_\chi}=d\,S(\overline{\chi^{*}})\neq 0$, where $\overline{\chi^{*}}$ is an
      odd primitive character modulo $f$ and $S$ is as in Lemma~\ref{lem:S}.
\end{enumerate}
\end{lemma}

\begin{proof}
First, $\chi^{*}(-1)=\chi(-1)=-1$ because $\gcd(-1,n)=1$; hence $\chi^{*}$ and
$\overline{\chi^{*}}$ are odd primitive characters modulo $f$.

(i) The map $\Z/f\Z\to\Z/n\Z$, $a\mapsto da$, is well defined and injective, since
$da\equiv da'\pmod{df}$ if and only if $a\equiv a'\pmod f$. Hence if
$x\equiv da\pmod n$, the class of $a$ modulo $f$ is uniquely determined by $x$, and both
$\gcd(a,f)$ and $\chi^{*}(a)$ depend only on this class. Thus $e_\chi$ is well
defined; and $e_\chi(d)=\chi^{*}(1)=1\neq 0$.

(ii) Let $\Sigma=\{\,da\bmod n:\ \gcd(a,f)=1\,\}\subseteq\Z/n\Z$ be the support of
$e_\chi$. Let $h\in\Gn$ be represented by an integer with $\gcd(h,n)=1$; then
$\gcd(h,f)=1$ because $f\mid n$. For $x=da\in\Sigma$ we get $hx\equiv d(ha)\pmod n$ with
$\gcd(ha,f)=1$, so $h\Sigma\subseteq\Sigma$; applying this to $h^{-1}$ gives
$h\Sigma=\Sigma$. Now let $x\in\Z/n\Z$. If $x=da\in\Sigma$, then
\[
(\rho(h)e_\chi)(x)=e_\chi\bigl(d(ha)\bigr)=\chi^{*}(ha)=\chi^{*}(h)\chi^{*}(a)
=\chi(h)\,e_\chi(x),
\]
using $\chi^{*}(h)=\chi(h)$. If $x\notin\Sigma$, then $hx\notin\Sigma$, so
$(\rho(h)e_\chi)(x)=e_\chi(hx)=0=\chi(h)e_\chi(x)$. Hence
$\rho(h)e_\chi=\chi(h)e_\chi$.

(iii) Since $e_\chi$ vanishes off $\Sigma$, and each $x\in\Sigma$ is of the form
$x\equiv da$ for a unique class $a\in(\Z/f\Z)^{\times}$, which we represent by an integer
$1\le a\le f-1$ with $\gcd(a,f)=1$,
\[
\ip{q_n}{e_\chi}=\sum_{\substack{1\le a\le f-1\\ \gcd(a,f)=1}}q_n(da)\,\overline{\chi^{*}(a)} .
\]
For such $a$ the integer $da$ satisfies $1\le d\le da\le d(f-1)=n-d\le n-1$, so by the
definition of $q_n$ we have $q_n(da)=da-\frac n2=d\bigl(a-\frac f2\bigr)$. Since
$\overline{\chi^{*}}$ is non-principal, $\sum_{a}\overline{\chi^{*}(a)}=0$. Therefore
\[
\ip{q_n}{e_\chi}
=d\sum_{\substack{1\le a\le f-1\\ \gcd(a,f)=1}}\Bigl(a-\frac f2\Bigr)\overline{\chi^{*}(a)}
=d\sum_{\substack{1\le a\le f-1\\ \gcd(a,f)=1}}a\,\overline{\chi^{*}(a)}
=d\,S(\overline{\chi^{*}}),
\]
which is non-zero by Lemma~\ref{lem:S} applied to $\psi=\overline{\chi^{*}}$.
\end{proof}

\begin{proposition}\label{prop:indep}
Let $R$ be as in \eqref{eq:R}. The vectors $\qq_h$, $h\in R$, are linearly independent
over $\C$ (hence over $\R$). Consequently
\[
\rank[\qq_h:h\in\Hn]=\frac{\varphi(n)}{2}.
\]
\end{proposition}

\begin{proof}
Suppose $\sum_{h\in R}c_h\,\qq_h=0$ with $c_h\in\C$. Define $c:\Gn\to\C$ by $c(h)=c_h$
and $c(-h)=-c_h$ for $h\in R$; this is well defined because $\Gn$ is the disjoint union
of the pairs $\{h,-h\}$, $h\in R$ (Lemma~\ref{lem:pair}). Applying $\iota^{-1}$ to the
relation and using \eqref{eq:ident} and \eqref{eq:odd}, we obtain in $V$
\[
v:=\sum_{h\in\Gn}c(h)\,\rho(h)q_n
=\sum_{h\in R}\bigl(c_h\,\rho(h)q_n+(-c_h)(-\rho(h)q_n)\bigr)
=2\sum_{h\in R}c_h\,\iota^{-1}(\qq_h)=0 .
\]
Let $\chi$ be an odd character. By Lemma~\ref{lem:rep} and Lemma~\ref{lem:eig}(ii), for
every $h\in\Gn$,
\[
\ip{\rho(h)q_n}{e_\chi}=\ip{q_n}{\rho(h^{-1})e_\chi}=\ip{q_n}{\chi(h^{-1})e_\chi}
=\overline{\chi(h^{-1})}\,\ip{q_n}{e_\chi}=\chi(h)\,\ip{q_n}{e_\chi}.
\]
Hence
\[
0=\ip{v}{e_\chi}=\sum_{h\in\Gn}c(h)\,\chi(h)\,\ip{q_n}{e_\chi}
=\ip{q_n}{e_\chi}\,\hat c(\chi),
\]
and since $\ip{q_n}{e_\chi}\neq 0$ by Lemma~\ref{lem:eig}(iii), we get
$\hat c(\chi)=0$. If $\chi$ is even, then
\[
\hat c(\chi)=\sum_{h\in R}\bigl(c_h\chi(h)-c_h\chi(-h)\bigr)
=\sum_{h\in R}c_h\bigl(\chi(h)-\chi(-1)\chi(h)\bigr)=0 .
\]
Thus $\hat c(\chi)=0$ for all $\chi\in\Ghat$, and Lemma~\ref{lem:char}(b) gives
$c\equiv 0$, i.e.\ $c_h=0$ for all $h\in R$.

For the last assertion, Lemma~\ref{lem:centered}(iii) shows
$\spn\{\qq_h:h\in\Hn\}=\spn\{\qq_h:h\in R\}$, which has dimension $|R|=\varphi(n)/2$ by
what we have just proved.
\end{proof}

\section{Proof of \texorpdfstring{Theorem~\ref{thm:main}}{the main theorem}}

\begin{proof}[Proof of Theorem~\ref{thm:main}]
By Lemma~\ref{lem:pair}, $\NN_n=\uu_h+\uu_{n-h}\in\col(\UU_n)$ for any $h\in\Hn$. Hence
$\qq_h=\uu_h-\frac12\NN_n\in\col(\UU_n)$ for every $h\in\Hn$; conversely
$\uu_h=\frac12\NN_n+\qq_h$. Therefore
\[
\col(\UU_n)=\R\NN_n+Q,\qquad Q:=\spn_{\R}\{\qq_h:h\in\Hn\}.
\]
By Lemma~\ref{lem:centered}(ii), the $n$-th coordinate of every vector of $Q$ is $0$,
whereas $(\NN_n)_n=2n\neq 0$. Hence $\NN_n\notin Q$ and
$\dim(\R\NN_n+Q)=\dim Q+1$. By Proposition~\ref{prop:indep} (together with the
convention on the ground field), $\dim Q=\varphi(n)/2$. Thus
\[
\rank(\UU_n)=\dim\col(\UU_n)=\frac{\varphi(n)}{2}+1,
\]
which also shows that the upper bound of Corollary~\ref{cor:upper} is attained.
\end{proof}

\section{Remarks}

\begin{remark}[On the support of \texorpdfstring{$e_\chi$}{e\_chi}]
It is essential in Lemma~\ref{lem:eig} that $e_\chi$ is supported on the set
$\{x:\gcd(x,n)=d\}$ (which consists of non-units when $f<n$) rather than on the units.
If one only used the rows $i$ with $\gcd(i,n)=1$, the relevant pairing would be
$\sum_{h\in\Gn}q_n(h)\chi(h)=nB_{1,\chi}$, and for a character $\chi$ modulo $n$ of
conductor $f$ one has $B_{1,\chi}=B_{1,\chi^{*}}\prod_{p\mid n,\;p\nmid f}\bigl(1-\chi^{*}(p)\bigr)$
(see \cite[Chapter~4]{Wa}), which may vanish. For example, let $n=21$ and let $\chi$ be
induced by the quadratic character modulo $3$, so that $\chi^{*}(7)=1$. The units modulo
$21$ that are $\equiv 1\pmod 3$ are $1,4,10,13,16,19$, and those $\equiv 2\pmod 3$ are
$2,5,8,11,17,20$; both sets have sum $63=6\cdot\frac{21}{2}$, so
\[
\sum_{h\in G_{21}}q_{21}(h)\chi(h)
=\sum_{h\equiv 1\,(3)}\Bigl(h-\frac{21}{2}\Bigr)-\sum_{h\equiv 2\,(3)}\Bigl(h-\frac{21}{2}\Bigr)=0 .
\]
Accordingly, the $12\times 6$ submatrix of $[\qq_h:h\in R]$ formed by the rows $i$ with
$\gcd(i,21)=1$ has rank only $5$, whereas the full $21\times 6$ matrix has rank $6$:
the rows with $\gcd(i,21)\in\{3,7\}$ are needed, and the eigenvector $e_\chi$ of
Lemma~\ref{lem:eig} (with $d=7$, $f=3$) is supported precisely on the rows $7$ and $14$.
\end{remark}

\begin{remark}[Relation to the projection argument]
The proof of Proposition~\ref{prop:indep} can be rephrased in terms of the orthogonal
projections
\[
P_\chi=\frac{1}{\varphi(n)}\sum_{h\in\Gn}\overline{\chi(h)}\,\rho(h)
\qquad\text{onto}\qquad
V_\chi=\{v\in V:\rho(h)v=\chi(h)v \text{ for all } h\in\Gn\}.
\]
One has $V=\bigoplus_{\chi}V_\chi$ orthogonally, $P_\chi q_n\in\spn\{\rho(h)q_n:h\in\Gn\}$,
and $\ip{P_\chi q_n}{e_\chi}=\ip{q_n}{e_\chi}\neq 0$ for odd $\chi$, because
$e_\chi\in V_\chi$ and $P_\chi$ is self-adjoint. The $\varphi(n)/2$ non-zero vectors $P_\chi q_n$
($\chi$ odd) lie in mutually orthogonal subspaces and are therefore linearly independent.
The Fourier-inversion argument given above is the same computation, organized so as to
prove directly that the columns $\qq_h$, $h\in R$, are linearly independent.
\end{remark}

\begin{remark}[The hypothesis \texorpdfstring{$n>2$}{n>2}]
The hypothesis $n>2$ is used exactly to guarantee $-1\neq 1$ in $\Gn$
(Lemma~\ref{lem:pair} and Lemma~\ref{lem:char}(c)). It cannot be dropped: for $n=2$ one
has $\mathcal H_2=\{1\}$, $\UU_2=(1,2)^{\top}$ and $\rank(\UU_2)=1\neq\varphi(2)/2+1$.
Theorem~\ref{thm:main} makes no distinction between odd, singly even and doubly even
sample sizes $n$.
\end{remark}

\section*{Acknowledgments}
The author is deeply grateful to Professor Kai-Tai Fang for introducing him to the fields of quasi-Monte Carlo methods and representative points. Professor Fang’s pioneering contributions to uniform design have served as a significant source of inspiration for the present work.
The author also gratefully acknowledges the assistance of ChatGPT (GPT-6, OpenAI) and Claude (Anthropic) in the preparation of this manuscript.


\begin{thebibliography}{9}
\bibitem[Ap]{Ap} T.~M. Apostol, \emph{Introduction to Analytic Number Theory},
Undergraduate Texts in Mathematics, Springer, New York, 1976.
\bibitem[Da]{Da} H.~Davenport, \emph{Multiplicative Number Theory}, 3rd ed.,
Graduate Texts in Mathematics 74, Springer, New York, 2000.
\bibitem[Wa]{Wa} L.~C. Washington, \emph{Introduction to Cyclotomic Fields}, 2nd ed.,
Graduate Texts in Mathematics 83, Springer, New York, 1997.
\end{thebibliography}
\end{document}